\documentclass[12pt,a4paper,reqno]{amsart}
\usepackage{amsmath}
\usepackage{amsfonts}
\usepackage{amssymb}
\usepackage{bm} 

\newcommand\R{{\mathbf{R}}}

\renewcommand\H{{\mathbf{H}}}

\renewcommand\S{{\mathcal{S}}}
\newcommand\Sch{{\operatorname{Schwartz}}}
\newcommand\T{{\mathrm{T}}}
\newcommand\E{{\mathrm{E}}}

\newcommand\Sym{{\operatorname{Sym}}}
\newcommand\Dil{{\operatorname{Dil}}}
\newcommand\Trans{{\operatorname{Trans}}}
\newcommand\Time{{\operatorname{Time}}}
\newcommand\Rev{{\operatorname{Rev}}}
\newcommand\Rot{{\operatorname{Rot}}}
\newcommand\Energy{{\dot{\mathcal{H}^1}}}

\newcommand\eps{{\varepsilon}}

\theoremstyle{plain}
  \newtheorem{theorem}[subsection]{Theorem}
  \newtheorem{conjecture}[subsection]{Conjecture}
  \newtheorem{claim}[subsection]{Claim}
  \newtheorem{proposition}[subsection]{Proposition}
  \newtheorem{lemma}[subsection]{Lemma}
  \newtheorem{corollary}[subsection]{Corollary}

\theoremstyle{remark}
  \newtheorem{remark}[subsection]{Remark}

\theoremstyle{definition}
  \newtheorem{definition}[subsection]{Definition}

\include{psfig}

\begin{document}

\title[Global regularity of wave maps III]{Global regularity of wave maps III.  Large energy from $\R^{1+2}$ to hyperbolic spaces}
\author{Terence Tao}
\address{Department of Mathematics, UCLA, Los Angeles CA 90095-1555}
\email{ tao@@math.ucla.edu}

\vspace{-0.3in}
\begin{abstract}
We show that wave maps $\phi$ from two-dimensional Minkowski space $\R^{1+2}$ to hyperbolic spaces $\H^m$ are globally smooth in time if the initial data is smooth, conditionally on some reasonable claims concerning the local theory of such wave maps, as well as the self-similar and travelling (or stationary solutions); we will address these claims in the sequels \cite{tao:heatwave2}, \cite{tao:heatwave3}, \cite{tao:heatwave4}, \cite{tao:heatwave5} to this paper.  Following recent work in critical dispersive equations, the strategy is to reduce matters to the study of an \emph{almost periodic} maximal Cauchy development in the energy class.  We then repeatedly analyse the stress-energy tensor of this development (as in \cite{grillakis-energy}, \cite{tao:forges}) to extract either a self-similar, travelling, or degenerate non-trivial energy class solution to the wave maps equation.  We will then rule out such solutions in the sequels to this paper, establishing the desired global regularity result for wave maps.
\end{abstract}

\maketitle

\section{Introduction}

\subsection{Wave maps}

This paper is concerned with the global regularity problem for two-dimensional (i.e. energy-critical) wave maps into hyperbolic spaces.  To explain this problem, we first need to introduce some standard notation.

For any $n \geq 1$ let $\R^{1+n}$ be Minkowski space $\{ (t,x): t \in \R, x \in \R^n \}$
with the usual metric $g_{\alpha \beta} x^\alpha x^\beta := -dt^2 + dx^2$.  We shall work primarily in $\R^{1+2}$, parameterised by Greek indices $\alpha, \beta = 0,1,2$, raised and lowered in the usual manner.  We also parameterise the spatial coordinates by Roman indices $i,j = 1,2$.  We let $SO(n,1)$ be the space of Lorentz transformations on $\R^{1+n}$ (i.e. unimodular linear isometries of Minkowski space).

Fix $m \geq 1$.  Let \emph{hyperbolic space} $\H = (\H^m,h)$ be the simply-connected $m$-dimensional Riemannian manifold of constant negative sectional curvature $-1$.  The precise realisation of hyperbolic space is not important for our purposes, but for sake of concreteness one can take
$$ \H := \{ (t,x) \in \R^{1+m}: t = +\sqrt{1 + |x|^2} \} \subset \R^{1+m}$$
to be the upper sheet of the unit hyperboloid in Minkowski space $\R^{1+m}$ with $h$ being the induced Riemannian metric.  Note that $SO(m,1)$ acts transitively on $\H$ with stabiliser equal to the orthogonal group $SO(m)$, thus $\H \equiv SO(m,1) / SO(m)$.

Define a \emph{classical wave map} to be a pair $\phi = (\phi,I)$, where $I \subset \R$ is an interval, and $\phi: I \times \R^2 \to \H$ is a smooth map on the slab $I \times \R^2 \subset \R^{1+2}$ which differs\footnote{This is slightly more general than the notion of a classical wave map in some earlier papers (including some of the author), in which the wave map is assumed to be smooth and compactly supported modulo constants rather than Schwartz modulo constants.  It will be more convenient for us to work in the Schwartz category as this will be preserved by harmonic map heat flow.  However the distinctions between the two spaces are quite minor.} from a constant $\phi(\infty) \in \H$ by a function Schwartz in space (embedding $\H$ in $\R^{1+m}$ to define the Schwartz space), and is a (formal)  critical point of the Lagrangian
\begin{equation}\label{lagrangian}
 \int_{\R^{1+2}} \langle \partial^\alpha \phi(t,x), \partial_\alpha \phi(t,x) \rangle_{h(\phi(t,x))}\ dt dx
\end{equation}
in the sense that $\phi$ obeys the corresponding Euler-Lagrange equation
\begin{equation}\label{cov}
 (\phi^*\nabla)^\alpha \partial_\alpha \phi = 0,
\end{equation}
where $(\phi^*\nabla)^\alpha$ is covariant differentiation on the vector bundle $\phi^*(T\H)$ 
with respect to the pull-back $\phi^*\nabla$ via $\phi$ of the Levi-Civita connection $\nabla$ on the tangent bundle $T\H$ of $\H$.  If $I = \R$, we say that the wave map is \emph{global}.  

\begin{remark}
The wave map equation is closely related to \emph{$\sigma$-models} in gauge theory, and is the Minkowski analogue of a harmonic map.
\end{remark}

We record five important (and well known) symmetries of wave maps:

\begin{itemize}
\item For any $t_0 \in \R$, we have the time translation symmetry 
\begin{equation}\label{time-trans}
\Time_{t_0}: \phi(t,x) \mapsto \phi(t-t_0,x).
\end{equation}
\item For any $x_0 \in \R^2$, we have the space translation symmetry
\begin{equation}\label{space-trans}
\Trans_{x_0}: \phi(t,x) \mapsto \phi(t,x-x_0).
\end{equation}
\item Time reversal symmetry
\begin{equation}\label{time-reverse}
\Rev: \phi(t,x) \mapsto \phi(-t,x).
\end{equation}
\item For every $U \in SO(m,1)$, we have the target rotation symmetry
\begin{equation}\label{rotate}
\Rot_U: \phi(t,x) \mapsto U \circ \phi(t,x).
\end{equation}
\item For every $\lambda > 0$, we have the scaling symmetry
\begin{equation}\label{cov-scaling}
\Dil_\lambda: \phi(t,x) \mapsto \phi(\frac{t}{\lambda}, \frac{x}{\lambda})
\end{equation}
\end{itemize}
We will exploit all of these symmetries in the sequel.  (Lorentz invariance will not be directly used in this paper, although it will be lurking behind the scenes in some of our stress-energy manipulations.)

Another symmetry of the wave maps equation is the diffeomorphism symmetry on the domain $(\R^{1+2},g)$ (generalising \eqref{time-trans}, \eqref{space-trans}, \eqref{time-reverse}, and the Lorentz symmetry). We shall (implicitly) exploit this symmetry via the conservation law
\begin{equation}\label{conserv}
 \partial^\alpha \T_{\alpha \beta} = 0
\end{equation}
for the \emph{stress-energy tensor} $\T: \R^{1+2} \to \Sym^2(\R^{1+2})$ defined by
\begin{equation}\label{stress-def}
\T_{\alpha \beta} := \langle \partial_\alpha \phi, \partial_\beta \phi \rangle_{h(\phi)}
- \frac{1}{2} g_{\alpha \beta} \langle \partial^\gamma \phi, \partial_\gamma \phi \rangle_{h(\phi)}.
\end{equation}
In particular, the \emph{energy} $\E(\phi)$, defined by
\begin{equation}\label{energy-def}
 \E(\phi) = \E(\phi[t]) := \int_{\R^2} \T_{00}(t,x)\ dx = \int_{\R^2} \frac{1}{2} |\partial_t \phi|_{h(\phi)}^2
+ \frac{1}{2} |\nabla_x \phi|_{h(\phi)}^2\ dx
\end{equation}
is conserved in time.  From \eqref{conserv} we also see that the energy is preserved under all the symmetries \eqref{time-trans}-\eqref{cov-scaling} mentioned above; the fact that it is invariant under the scaling symmetry is specific to two spatial dimensions, which is thus the \emph{critical} dimension for this problem.

We also note that the stress-energy tensor $\T_{\alpha \beta}$ in two spatial dimensions determines all inner products $\langle \partial_\alpha \phi, \partial_\beta \phi \rangle_{h(\phi)}$ by the identity
\begin{equation}\label{destress}
\langle \partial_\alpha \phi, \partial_\beta \phi \rangle_{h(\phi)} = \T_{\alpha \beta} -
g_{\alpha \beta} \operatorname{tr}(\T)
\end{equation}
where of course $\operatorname{tr}(\T) := g^{\alpha \beta} \T_{\alpha \beta}$.  The formula \eqref{destress} is important for us, as it allows us to manipulate quadratic expressions in the derivatives of $\phi$ for an abstract class of ``energy class solutions'', for which $\phi$ is not necessarily defined in a classical sense, but for which one can still meaningfully define a stress-energy tensor.

All of the above symmetries and conservation laws will play an important role in the analysis which follows.

\subsection{The global regularity conjecture}

Define \emph{classical data} to be any pair $\Phi_0 := (\phi_0,\phi_1)$, where $\phi_0: \R^2 \to \H$ is a smooth map which differs from a constant by a Schwartz function, and $\phi_1: \R^2 \to T\H$ is a Schwartz map with $\phi_1(x) \in T_{\phi_0(x)} \H$ for all $x \in \R^2$; let $\S$ denote the space of all classical data, equipped with the Schwartz  topology; one can view this space as a nonlinear analogue of the Schwartz space $\Sch(\R^2) \times \Sch(\R^2)$.  Observe that for any classical wave map $(\phi,I)$ and any time $t \in I$, the pair $\phi[t] := (\phi(t), \partial_t \phi(t))$ lies in $\S$.  Also observe that spatial translations \eqref{space-trans}, time reflection \eqref{time-reverse}, target rotation \eqref{rotate}, and scaling \eqref{cov-scaling} act on $\S$ in the obvious manner, namely
\begin{align}
\Trans_{x_0}: (\phi_0(x), \phi_1(x)) &\mapsto (\phi_0(x-x_0), \phi_1(x-x_0)) \label{space-trans-data}\\
\Rev: (\phi_0(x), \phi_1(x)) &\mapsto (\phi_0(x), -\phi_1(x)) \label{time-reverse-data}\\
\Rot_U: (\phi_0(x), \phi_1(x)) &\mapsto (U\phi_0(x), dU(\phi_0(x))(\phi_1(x))) \label{rotate-data}  \\
\Dil_\lambda: (\phi_0(x), \phi_1(x)) &\mapsto (\phi_0(\frac{x}{\lambda}), \frac{1}{\lambda} \phi_1(\frac{x}{\lambda})) \label{scaling-data}
\end{align}
respectively.  (The (nonlinear) action of time translation \eqref{time-trans} can be defined conditionally on Conjecture \ref{conj} below.)  Also observe that the stress-energy tensor $\T$ defined in \eqref{stress-def} can be viewed as a continuous map from $\S$ to the space $L^1( \R^2 \to \Sym^2(\R^{1+2}) )$ of absolutely integrable symmetric rank two tensors on $\R^2$.
 
We will be concerned with the following conjecture, which can be found for instance in \cite{klainerman:unified}:

\begin{conjecture}[Global regularity for wave maps]\label{conj}  For every $\Phi_0 \in \S$ there exists a unique global classical wave map $(\phi,\R)$ with $\phi[0] = \Phi_0$.
\end{conjecture}

\begin{remark}
This conjecture forms part of an intensive study of the initial value problem for the wave maps equation; see \cite{kman.barrett}, \cite{kman.selberg:survey}, \cite{shatah-struwe}, \cite{struwe.barrett}, \cite{tataru:survey}, \cite[Chapter 6]{tao:cbms}, \cite{rod}, \cite{krieger:survey} for surveys of this problem.  By using finite speed of propagation (and the contractibility of the target manifold $\H$), one can remove the support hypotheses on the data and solution, allowing one to construct global smooth solutions to the wave maps equation from arbitrary smooth initial data, but it will be convenient for us to retain the support conditions in the sequel.  One can also generalise from $\H$ to other manifolds of constant negative sectional curvature (which we can normalise to be $-1$) by the usual lifting argument (taking advantage of the contractibility of the domain $\R^2$).
\end{remark}

We briefly summarise some of the past progress on this conjecture (see \cite[Chapter 6]{tao:cbms} or \cite{krieger:survey} for a more detailed survey).  From classical energy methods (and expressing the wave maps equation in local coordinate charts) one can easily establish the existence of a \emph{local} classical solution, and to establish uniqueness of any (local or global) classical solution; see e.g. \cite{shsw:wavemap}.  Local solutions can also be constructed in significantly rougher function spaces than $C^\infty$: see \cite{kman.selberg}, \cite{tataru:wave2}, \cite{tataru:wave3}, \cite{tataru:survey}.  The conjecture has been established for radially symmetric data in \cite{christ.spherical.wave}, for rotationally equivariant data in \cite{shatah.shadi.eqvt}, and for various classes of small data (or perturbations of symmetric data) in \cite{sideris.harmonic}, \cite{tataru:wave2}, \cite{krieger:2d}, \cite{krieger:stab}; also, global \emph{weak} solutions were constructed in \cite{muller}. If the hyperboloid is replaced by the sphere, analogues of all the above results hold (for instance, global regularity in the small energy case was established in \cite{tao:wavemap2}), but the analogue of the above conjecture fails \cite{kst}, \cite{rs} (see also the numerical work in \cite{bct}, \cite{isenberg}). 

Since the initial release of this preprint, an alternate proof of this conjecture has been provided by Sterbenz and Tataru \cite{sterbenz}, \cite{sterbenz2}, as well as Krieger and Schlag (private communication).

\begin{remark}
As mentioned earlier, the dimension $d=2$ of space is the critical dimension for this problem. In the sub-critical dimension $d=1$, global regularity can be easily established \cite{gu}, \cite{lady}, \cite{shatah}, \cite{keel:wavemap}, whereas in super-critical dimensions $d > 2$, singularities are expected to form \cite{css}.  On the other hand, the local and perturbative theory for $d=2$ has analogues for $d>2$ (replacing the energy norm by the critical Sobolev norm $\dot H^{d/2} \times \dot H^{d/2-1}$), and in fact the theory simplifies in this setting, particularly for $d \geq 4$ or $d \geq 5$: see \cite{kl-mac:duke-null}, \cite{kr:wavemap}, \cite{krieger:3d}, \cite{nahmod}, \cite{shsw:wavemap}, \cite{tao:wavemaps},  \cite{tao:wavemap2}, \cite{tataru:wave1}, \cite{tataru:wave2}.
\end{remark}

\subsection{Claims}

The purpose of this paper is to establish Conjecture \ref{conj}, conditionally on five simpler (and plausible) claims, which we shall address in sequels \cite{tao:heatwave2}, \cite{tao:heatwave3}, \cite{tao:heatwave4}, \cite{tao:heatwave5} to this paper.  We now state these claims.

The first claim asserts the existence of an energy space $\Energy$, which is a nonlinear analogue of the scalar energy space $\dot H^1(\R^2) \times L^2(\R^2)$, with which to hold generalisations of the classical data $\Phi_0$ appearing above.  This energy space will then be used in all the other claims.

\begin{claim}[Energy space]\label{energy-claim}  There exists a complete metric space $\Energy$ with a continuous map $\iota: \S \to \Energy$, that obeys the following properties:
\begin{itemize}
\item[(i)] $\iota(\S)$ is dense in $\Energy$.
\item[(ii)] $\iota$ is invariant under the action \eqref{rotate-data} of the rotation group $SO(m,1)$, thus $\iota(\Rot_U \Phi) = \iota(\Phi)$ for all $\Phi \in \S$. Conversely, if $\iota(\Phi) = \iota(\Psi)$, then $\Psi = \Rot_U(\Phi)$ for some $U \in SO(m,1)$.
\item[(iii)] The actions \eqref{space-trans-data}, \eqref{time-reverse-data}, \eqref{scaling-data} on $\S$ extend to a continuous isometric action on $\Energy$ (after quotienting out by rotations as in (ii)).
\item[(iv)] The stress-energy map $\T: \S \to L^1(\R^2 \to \hbox{Sym}^2(\R^{1+2}))$ extends to a continuous map $\T: \Energy \to L^1(\R^2 \to \hbox{Sym}^2(\R^{1+2}))$ (again after quotienting out by rotations as in (ii)).  In particular, we have a continuous energy functional $\E: \Energy \to [0,+\infty)$.
\item[(v)] If $\Phi \in \Energy$ has zero energy, thus $\E(\Phi)=0$, then $\Phi$ is constant (or more precisely, $\Phi = \iota(p,0)$ for any $p \in \H$). 
\end{itemize}
\end{claim}

\begin{remark}
Morally speaking, $\Energy$ is the space of ``all'' pairs $\Phi_0 = (\phi_0,\phi_1)$ with $\phi_0: \R^2 \to \H$ and $\phi_1: \R^2 \to T\H$, with $\phi_1(x) \in T_{\phi_0(x)} \H$ for all (or almost all) $x$, and with the energy
$$ \E( \Phi_0 ) := \int_{\R^2} \frac{1}{2} |\nabla \phi_0|_{h(\phi_0)}^2 + \frac{1}{2} |\phi_1|_{h(\phi_0)}^2\ dx$$
finite, where the derivative $\nabla \phi_0$ is interpreted in some suitably weak sense, and then quotiented out by the action \eqref{rotate}.  There are however some technical difficulties arising from the fact that $\phi_0$ takes values in a manifold $\H$ rather than in a vector space, and also because the energy norm in two dimensions is not expected to control boundedness or continuity of functions in the energy space.  Note that $\H$ cannot be isometrically embedded into a Euclidean space, and using coordinate patches to try to define the energy space is also problematic if one wants to retain the rotation invariance (ii).  Also, the metric structure one should select on the nonlinear space $\Energy$ is not immediately obvious\footnote{By our conventions, any classical data $\Phi_0$ will have a zero separation from its rotations $\Rot_U \Phi_0$; this is the analogue of the fact that two scalar functions on $\R^2$ which differ by a constant have a zero separation in the $\dot H^1(\R^2)$ (semi-)metric; this is related to the difficulty that the $\dot H^1(\R^2)$ norm fails to control the $L^\infty(\R^2)$ norm.}.
\end{remark}

In \cite{tao:heatwave2} we shall construct an energy space $\Energy$ as in Claim \ref{energy-claim}, by taking a metric completion of the classical data space $\S$, using the harmonic map heat flow to create a ``nonlinear Littlewood-Paley resolution'' with which one can measure distances.  We will then establish the following large-data local well-posedness claim in this space:

\begin{claim}[Large data local-wellposedness in the energy space]\label{lwp-claim}  For every time $t_0 \in \R$ and every initial data $\Phi_0 \in \Energy$ there exists a \emph{maximal lifespan} $I \subset \R$, and a \emph{maximal Cauchy development} $\phi: t \mapsto \phi[t]$ from $I \to \Energy$, which obeys the following properties:
\begin{itemize}
\item[(i)] (Local existence) $I$ is an open interval containing $t_0$.
\item[(ii)] (Strong solution) $\phi: I \to \Energy$ is continuous.
\item[(iii)] (Persistence of regularity) If $\Phi_0 = \iota( \tilde \Phi_0 )$ for some classical data $\tilde \Phi_0$, then there exists a classical wave map $(\tilde \phi,I)$ with initial data $\tilde \phi[0] = \tilde \Phi_0$ such that $\phi[t] = \iota(\tilde \phi[t])$ for all $t \in I$.  (Note that $\tilde \Phi_0$, $\tilde \phi$ are ambiguous up to the action of the rotation group $SO(m,1)$, but this ambiguity disappears after applying $\iota$.)
\item[(iv)] (Continuous dependence)  If $\Phi_{0,n}$ is a sequence of data in $\Energy$ converging to a limit $\Phi_{0,\infty}$, and $\phi_n: I_n \to \Energy$ and $\phi_\infty: I_\infty \to \Energy$ are the associated maximal Cauchy developments on the associated maximal lifespans, then for every compact subinterval $K$ of $I_\infty$, we have $K \subset I_n$ for all sufficiently large $n$, and $\phi_n$ converges uniformly to $\phi$ on $K$ in the $\Energy$ topology.
\item[(v)] (Maximality)  If $t_* \in \R$ is a finite endpoint of $I$, then $\phi(t)$ has no convergent subsequence in $\Energy$ as $t \to t_*$.
\end{itemize}
\end{claim}

\begin{remark}
This type of claim has already been established in \cite{tataru:wave3} (see also \cite{tataru:survey}) when the target manifold $\H$ is replaced by one which can be uniformly embedded into a Euclidean space.  In the case of small energy, a relevant regularity result in this direction appears in \cite{krieger:2d} (see also \cite{tao:wavemap2} for an analogous result when the target is a sphere).  
If the energy space is replaced by a slightly smaller Besov space (which, in contrast to the energy space, does control the boundedness and continuity of solutions), then this claim essentially appears in \cite{tataru:wave2}.  For semilinear equations for which solutions at the critical regularity can be constructed by iteration, results such as Claim \ref{lwp-claim} are standard consequences of the function space estimates used in the iteration scheme; see e.g. \cite{tao:cbms}.
\end{remark}

\begin{remark}
The author does not claim that these maximal Cauchy developments solve the wave maps equation \eqref{cov} in any weak (distributional) sense, because the author does not know the most suitable way in which to set up the theory of distributions for finite energy functions taking values\footnote{This problem can be partly resolved by selecting an orthonormal frame and working purely with the derivative map, as in \cite{shsw:wavemap}, \cite{nahmod}; but then one has to quotient out by the gauge invariance, which causes a separate set of technical issues to arise.  These issues can mostly be resolved by using the \emph{caloric gauge} from \cite{tao:forges}; see \cite{tao:heatwave3} for further discussion.} in hyperbolic space.  However, Claim \ref{energy-claim} and Claim \ref{lwp-claim} do allow us to approximate any of these abstract maximal Cauchy developments as the limit of smooth solutions (although the limit is in the sense of uniform convergence in the abstract space $\Energy$), which turns out to be sufficient for our purposes\footnote{As a rule of thumb, we can manipulate energy class solutions as if they were classical as long as one only considers \emph{first} derivatives of the solution and not second derivatives.  In particular, one cannot directly access the wave maps equation \eqref{cov}.  However, the stress-energy $\T$ is well-defined as a distribution, and so one can still utilise the conservation law \eqref{conserv} in the sense of distributions, which serves as a (weak) substitute for \eqref{cov}.}.  In particular, many conservation laws, symmetries, and Morawetz-type estimates which can be rigorously derived for classical solutions, will extend to these maximal Cauchy developments by standard limiting arguments.  For instance, since the energy \eqref{energy-def} is preserved for classical solutions, we see from Claim \ref{lwp-claim} and Claim \ref{energy-claim} that it is also preserved for maximal Cauchy developments, and so we can meaningfully refer to the energy or momentum of a Cauchy development.  As we shall see in Lemma \ref{Maxcauchy} below, the concept of a maximal Cauchy development can be defined independently of any choice of initial time or initial data, and so we shall often refer to such developments without mention of initial conditions.
\end{remark}

We shall establish Claim \ref{lwp-claim} in \cite{tao:heatwave3} using the harmonic heat flow (i.e. the caloric gauge) to renormalise the wave maps equation, as indicated in \cite{tao:forges}.  We remark that the non-positive curvature of $\H$ plays a key role here, in forcing the harmonic heat flow to converge to a constant map (see \cite{eells}). Note that Conjecture \ref{conj} would follow immediately from Claim \ref{lwp-claim}(iii) and the following conjecture:

\begin{conjecture}[Global well-posedness in the energy space]\label{conj2}  Every maximal Cauchy development in Claim \ref{lwp-claim} is global, i.e. the maximal lifespan $I$ is always equal to $\R$.
\end{conjecture}

We now turn to the most technically difficult of all the five claims to establish.  
We shall refer to the restriction of a maximal Cauchy development to a subinterval as a \emph{partial Cauchy development}, and refer to maximal or partial Cauchy developments as \emph{energy class solutions}; as mentioned earlier, the reader is encouraged to view such solutions as abstract limits of smooth solutions.

\begin{definition}[Almost periodicity]\label{ap-def} An energy class solution $\phi: I \to \Energy$ is\footnote{Strictly speaking, we should refer to such solutions as ``almost periodic modulo dilation and translation symmetry'', but we shall omit the mention of the dilation and translation symmetries for brevity.} \emph{almost periodic} if there exists functions $N: I \to (0,+\infty)$ and $x: I \to \R^2$, and a compact set $K \subset \Energy$ such that
$$ \Dil_{N(t)} \Trans_{-x(t)} \phi[t] \in K$$
for all $t \in I$.  For instance, a (hypothetical) stationary wave map (i.e. a harmonic map), a travelling wave map (i.e. a Lorentz transform of a wave map) or self-similar wave map in the energy class would be almost periodic (in these cases, $K$ would just be a single point).  We refer to the functions $N = N_\phi$ and $x = x_\phi$ as the\footnote{These functions are not quite unique; for instance $N(t)$ is only determined up to multiplicative constants, and $x(t)$ is only determined up to an additive error of $O(1/N(t))$, but this lack of uniqueness will not be a concern for us.  Similarly, the compactness modulus $K$ is also not unique.} \emph{frequency scale function} and \emph{position function} of the almost periodic solution respectively.  The compact set $K$ will be referred to as the \emph{compactness modulus} of the solution $\phi$.
\end{definition}

\begin{remark} Informally, an almost periodic solution is concentrated spatially in the region $\{ x: x = x(t) + O( 1/N(t) ) \}$ and in frequency space in the region $\{ \xi: \xi = O( N(t) ) \}$ at each time $t \in I$, although it is not immediately obvious how to define the ``frequency'' of a map $\phi: \R^{1+2} \to \H$ into hyperbolic space.  See Lemma \ref{localise} for a partial formalisation of this intution in the physical (spatial) domain only.
\end{remark}

\begin{claim}[Minimal energy blowup solution]\label{minimal}  Suppose that Conjecture \ref{conj2} fails.  Then there exists an almost periodic maximal Cauchy development $\phi: I \to \Energy$ with non-zero energy.
\end{claim}

\begin{remark}
Analogous claims to Claim \ref{minimal} have been established for the energy-critical NLS in \cite{borg:scatter}, \cite{borg:book}, \cite{gopher}, \cite{RV}, \cite{thesis:art}, \cite{merlekenig}, for the mass-critical NLS in \cite{keraani}, \cite{compact}, for the energy-critical NLW in \cite{mkwave} (see also \cite{Evian}), and for the energy-critical and mass-critical Hartree equations in \cite{mxz1}, \cite{mxz2}, \cite{mxz3}, \cite{mxz4}.  The arguments in these papers rely on perturbation theory (generalising local well-posedness results such as Claim \ref{lwp-claim} to compare approximate solutions with exact solutions), as well as various spatial and frequency decompositions to split any solution which is not behaving almost periodically into two almost non-interacting components of strictly smaller energy; see \cite{tao:cdm} for further discussion\footnote{Intuitively, one can justify Claim \ref{minimal} as follows.  If Conjecture \ref{conj2} failed, then there should exist a \emph{critical energy} $E_{\operatorname{crit}} > 0$ below which the solution exists globally and obeys ``scattering'' estimates (e.g. are bounded in some suitable spacetime norm).  We then take a sequence of solutions with energy approaching $E_{\operatorname{crit}}$ which do not obey scattering estimates in the limit.  Such solutions must be localised in space or frequency, otherwise they could be decoupled into the superposition of two weakly interacting components of strictly smaller energy, which should contradict the definition of $E_{\operatorname{crit}}$.  This space and frequency localisation then leads to the desired almost periodicity in the limit.}  We will establish this claim in \cite{tao:heatwave4}, \cite{tao:heatwave5}, by refining the methods used to prove Claim \ref{lwp-claim} and establishing various non-linear analogues of spatial and frequency decompositions used in the earlier papers (which dealt with scalar models, in which standard decompositions such as Littlewood-Paley projections could be used without difficulty).
\end{remark}

In view of Claim \ref{minimal}, we see that to establish Conjecture \ref{conj}, it suffices to show that almost periodic maximal Cauchy developments of strictly positive energy do not exist.  Three possible candidates for such solutions would be 
\begin{itemize}
\item \emph{Stationary solutions} $\phi: \R \times \R^2 \to \H$ in which $\phi(t,x) = \phi(x)$;
\item More generally, \emph{travelling solutions} $\phi: \R \times \R^2 \to \H$ in which $\phi(t,x) = \phi(x-vt)$ for some velocity $v \in \R^2$; and
\item \emph{Self-similar solutions} $\phi: (0,+\infty) \times \R^2 \to \H$ in which $\phi(t,x) = \phi(x/t)$, and which are constant outside of the light cone $\{ (t,x): |x| \leq t \}$.  
\end{itemize}

Note that for physical reasons one expects finite energy travelling solutions to only travel at speeds $|v| < 1$ less than the speed of light $c=1$.

To extend such solution concepts from the classical context to the energy space context, we
use the stress-energy tensor $\T$:

\begin{definition}[Travelling and self-similar solutions]\label{travel}  An energy class solution $\phi: I \to \Energy$ is said to be \emph{travelling} with velocity $v \in \R^2$ if\footnote{Note that quadratic quantities such as the left-hand side of \eqref{static-def}, being defined from the stress-energy tensor $\T \in L^1(\R^2)$, are only defined up to almost everywhere equivalence, and so identities such as \eqref{static-def} should be understood to hold almost everywhere (or in the distributional sense) rather than everywhere.  We will not comment on this minor technical point further in this paper.} 
\begin{equation}\label{static-def}
|\partial_t \phi + v \cdot \nabla_x \phi|_{h(\phi)}^2 \equiv 0
\end{equation}
throughout $I \times \R^2$, where the quantity in \eqref{static-def} is defined via the stress-energy tensor $\T$ using \eqref{destress}.  Similarly, an energy class solution $\phi: I \to \Energy$ is said to be \emph{self-similar} if $\T = 0$ outside of the light cone $\{ (t,x): |x| \leq |t| \}$, and if
\begin{equation}\label{selfsimilar-def}
|t \partial_t \phi + x \cdot \nabla_x \phi|_{h(\phi)}^2 \equiv 0
\end{equation}
throughout $I \times \R^2$, where again we use \eqref{destress} to define the left-hand side of \eqref{selfsimilar-def} through the stress-energy tensor.
\end{definition}

\begin{remark} In the case of smooth solutions, the condition \eqref{static-def} is equivalent to the symmetry $\phi[t'] = \Trans_{v(t'-t)} \phi[t]$ for $t,t' \in I$, and similarly \eqref{selfsimilar-def} is equivalent to $\phi[t'] = \Dil_{t'/t} \phi[t]$ for $t,t' \in I$ with the same sign.  In particular, such solutions are almost periodic. However it is not immediately obvious whether the same equivalence (and almost periodicity) holds for solutions in the energy class.  Similarly, in the smooth case a travelling solution with $|v| < 1$ can be Lorentz transformed into a stationary solution (with $v=0$), but it is not immediately obvious that the same can be done for energy class solutions, even if they are global.
\end{remark}

\begin{claim}[No non-trivial self-similar or travelling energy class solutions]\label{selfsim} \  
\begin{itemize}
\item[(i)] The only almost periodic solutions $\phi: \R \to \Energy$ which are travelling with some velocity $v$ with $|v| < 1$ are the constant (i.e. zero-energy) solutions.
\item[(ii)] The only almost periodic solutions $\phi: (-\infty,0) \to \Energy$ which are self-similar are the constant solutions.
\end{itemize}
\end{claim}

\begin{remark} In the context of smooth finite energy solutions, the non-existence of non-trivial stationary solutions (i.e. travelling solutions with $v=0$) can be easily established from the \emph{Bochner-Weitzenb\"ock identity}
$$ \Delta |\nabla \phi|_{h(\phi)}^2 = 2 |\hbox{Hess}(\phi)|_{h(\phi)}^2 + \sum_{i,j} |\phi_i \wedge \phi_j|_{h(\phi)}^2 \geq 0$$
for such solutions, which relies crucially on the non-positive curvature of $\H$ to give both terms on the right-hand side the same sign.  The analogous claim for travelling solutions with $|v| < 1$ can then be deduced by the Lorentz invariance of the wave maps equation.  To rule out the self-similar solutions is a little trickier, and was done in \cite[Section 7.5]{shatah-struwe}, relying crucially on earlier work in \cite{lemaire}.  Morally speaking, these results already imply Claim \ref{selfsim}, but due to our abstract definition of an energy class solution one has to take a little care in adapting the previous arguments to establish this claim, which we do in \cite{tao:heatwave2}.
\end{remark}

Finally, we need a technical non-degeneracy claim about the energy space $\Energy$, which is needed to rule out the case of solutions travelling at the speed of light, $|v|=1$:

\begin{claim}[No non-trivial shift-invariant finite energy data]\label{nondeg}  Let $v \in \R^2$ be such that $|v|=1$, and let $\Phi = (\phi_0,\phi_1) \in \Energy$ be such that $|\phi_1 + v \cdot \nabla \phi_0|_{h(\phi)}^2, |w \cdot \nabla \phi_0|_{h(\phi_0)}^2 \equiv 0$ whenever $w \in \R^2$ is orthogonal to $v$, and this expression is defined in terms of the stress-energy tensor $\T(\Phi)$ using \eqref{destress}.  Then $\Phi$ has zero energy.
\end{claim}

\begin{remark} Formally, the hypotheses in Claim \ref{nondeg} imply that $\Phi$ has the shift invariance $\Trans_{sv}(\Phi) = \Phi$ for all $s \in \R$, and so the claim is essentially saying that there are no shift-invariant data in $\Energy$ other than the constant data.
\end{remark}

We will establish Claim \ref{nondeg} in \cite{tao:heatwave2}, using the ``non-linear Littlewood-Paley decomposition'' of $\Phi$ provided by the harmonic map heat flow.

\subsection{Main result}

We can now state the main result of this paper.

\begin{theorem}\label{main}  Assume that Claims \ref{energy-claim}, \ref{lwp-claim}, \ref{minimal}, \ref{selfsim}, \ref{nondeg} hold (with the same definition of energy space and maximal Cauchy development across these claims, of course).  Then Conjecture \ref{conj2} (and hence Conjecture \ref{conj}) holds.
\end{theorem}

Thus this paper, combined with the sequel papers \cite{tao:heatwave2}, \cite{tao:heatwave3}, \cite{tao:heatwave4}, \cite{tao:heatwave5}, gives a full proof of Conjecture \ref{conj}.

We now sketch how Theorem \ref{main} is to be proven.  Firstly, in view of Claim \ref{lwp-claim}, it suffices to show that the maximal Cauchy development of every initial data in $\Energy$ is global; applying Claim \ref{minimal}, we conclude that it suffices to show that no almost periodic maximal Cauchy development of strictly positive energy exists.

Suppose for contradiction that a almost periodic maximal Cauchy development $\phi: I \to \Energy$ with non-zero energy exists, with the attendant frequency scale function $N: I \to (0,+\infty)$ and position function $x: I \to \R^2$.  By a compactness argument exploiting the symmetries \eqref{time-trans}, \eqref{space-trans}, \eqref{time-reverse}, \eqref{cov-scaling} (and Claim \ref{lwp-claim}), it turns out that we can normalise so that $(-\infty,0] \subset I$, $N(0)=1$, $x(0)=0$, and $N(t) \leq 1$ for all $t \in (-\infty,0)$; we will refer to such solutions as \emph{normalised ancient solutions}.  

We will then use a Morawetz estimate from \cite{grillakis-energy} to conclude that such a normalised ancient solution exhibits some self-similar behaviour.  More precisely, there exists a sequence of times $t_n \to -\infty$ for which one has
\begin{equation}\label{Tmax}
 \int_{|x| \leq |t_n|} |\partial_t \phi + \frac{x}{t_n} \cdot \nabla_x \phi|_{h(\phi)}^2(t_n,x)\ dx \to 0
\end{equation}
(when interpreted suitably using the stress-energy tensor using \eqref{destress}).  Actually for technical reasons we need to modify \eqref{Tmax} by taking a (logarithmic) Hardy-Littlewood maximal function in time, but let us ignore this issue for now.  One can ensure (using the local theory) that $1/N(t_n) = O( |t_n| )$, and then also ensure (using finite speed of propagation) that $x(t_n) = O( |t_n| )$.   Passing to a subsequence, we may thus ensure that $1/N(t_n) = (\alpha + o(1)) |t_n|$ and $x(t_n) = (v+o(1)) t_n$ for some $0 \leq \alpha < \infty$ and $v \in \R^2$ with $|v| \leq 1$, where $o(1)$ denotes a quantity which goes to zero as $n \to \infty$.  We can now divide into three cases:

\begin{itemize}
\item[(i)] (Self-similar case) $\alpha \neq 0$.
\item[(ii)] (Non-self-similar timelike case) $\alpha=0$ and $|v| < 1$.
\item[(iii)] (Non-self-similar lightlike case) $\alpha=0$ and $|v|=1$.
\end{itemize}

In the self-similar case (i), we will use a rescaling and compactness argument (using the symmetry \eqref{cov-scaling} and Claim \ref{lwp-claim}) to extract a non-trivial self-similar solution (in the sense that \eqref{selfsimilar-def} vanishes).  But this will contradict Claim \ref{selfsim}(ii).

In the non-self-similar timelike case (ii), we will use another rescaling and compactness argument (using the symmetries \eqref{time-trans}, \eqref{space-trans}, \eqref{cov-scaling} and Claim \ref{lwp-claim}) to extract a non-trivial travelling solution (in the sense that \eqref{static-def} vanishes).  But this will contradict Claim \ref{selfsim}(i).

Finally, in the non-self-similar lightlike case we will again use rescaling and compactness (using the symmetries \eqref{time-trans}, \eqref{space-trans}, \eqref{cov-scaling} and Claim \ref{lwp-claim}), combined with some further stress-energy analysis, to obtain a non-trivial travelling solution which is degenerate in the sense that $|w \cdot \nabla \phi|_{h(\phi)}^2 \equiv 0$ whenever $w$ is orthogonal to $\phi$.  But this will contradict Claim \ref{nondeg}.

\begin{remark} Readers familiar with the literature may be surprised to see a lack of sophisticated function spaces (e.g. $X^{s,b}$ spaces and their refinements) and estimates in this paper, as well as a lack of discussion of gauges.  These aspects of the theory will however be present in abundance in \cite{tao:heatwave2}, \cite{tao:heatwave3}, \cite{tao:heatwave4}, \cite{tao:heatwave5} when the above five claims are proven, and will contribute significantly to the length of these papers.  Indeed, these papers will contain all the ``local'' or ``perturbative'' theory needed to prove Conjecture \ref{conj} (such as function space estimates, iteration arguments, and gauge fixing), whereas this paper is concerned primarily with the ``global'' or ``non-perturbative'' theory (and in particular, conservation laws and monotonicity formulae arising from an analysis of the stress-energy tensor).
\end{remark}

\begin{remark}
Amusingly, the arguments here are broadly parallel with those used by Perelman \cite{per1} to solve the Poincar\'e and geometrisation conjectures via Ricci flow; for instance, the normalised ancient solutions are analogous to the $\kappa$-solutions from \cite{per1}, while the stationary and self-similar solutions are analogous to the gradient shrinking solitons from \cite{per1}.  Of course, one major difficulty in \cite{per1} not present here is that for Ricci flow, there are non-trivial gradient shrinking solitons, in contrast to Claim \ref{selfsim}; this leads to singularities, surgeries, and other associated complications.  It may be that similar difficulties would arise if one were to generalise the arguments here to the positive curvature case.
\end{remark}

\subsection{Acknowledgements}

This project was started in 2001, while the author was a Clay Prize Fellow.  The author thanks Andrew Hassell and the Australian National University for their hospitality when a substantial portion of this work was initially conducted, and to Ben Andrews and Andrew Hassell for a crash course in Riemannian geometry and manifold embedding, and in particular to Ben Andrews for explaining the harmonic map heat flow.  The author also thanks Mark Keel for background material on wave maps, Daniel Tataru for sharing some valuable insights on multilinear estimates and function spaces, and to Igor Rodnianski for valuable discussions.  The author is supported by NSF grant DMS-0649473 and a grant from the Macarthur Foundation.

\section{Basic theory}

We begin with some basic consequences\footnote{See \cite{ktv} for the analogous theory in the case of the mass-critical NLS; note that for this ``soft'' aspect of the theory, which relies primarily on the local well-posedness theory and the symmetries, and to a lesser extent on the properties of the stress-energy tensor.} of Claims \ref{energy-claim} and Claims \ref{lwp-claim}, particularly as applied to almost periodic solutions; these claims will be assumed to be true throughout.  First of all, the maximal Cauchy development is indeed maximal (and unique):

\begin{lemma}[Maximality of Cauchy development]\label{Maxcauchy} Let $\phi: I \to \Energy$ be a maximal Cauchy development arising from an initial data $\Phi_0$ and time $t_0$, and let $\phi': I' \to \Energy$ be another maximal Cauchy development arising from another initial data $\Phi'_1$ and time $t'_0$.  Suppose that $\phi[t_1] = \phi'[t_1]$ for some $t_1 \in I \cap I'$.  Then $I=I'$ and $\phi=\phi'$.
\end{lemma}

\begin{proof}  It suffices to prove the claim in the case $t_1 = t_0$ or $t_1 = t'_0$, since the general case follows immediately from two applications of this special case (and Claim \ref{lwp-claim}).

Without loss of generality we can take $t_1 = t'_0$.  By Claim \ref{energy-claim}(i), we can find a sequence of classical initial data $\Phi_{0,n} \in \S$ which converge to $\Phi_0$ in $\Energy$ (after applying $\iota$).  By Claim \ref{lwp-claim} (iv), we have maximal Cauchy developments $\phi_n: I_n \times \R^2 \to \H$ with initial data $\phi_n[t_0] = \Phi_{0,n}$ which converge uniformly in $\Energy$ to $\phi$ on any compact subinterval $K$ of $I$ (with $n$ sufficiently large).  
In particular, $\phi_n[t'_0]$ converges in $\Energy$ to $\phi[t'_0] = \phi'[t'_0] = \Phi'_0$.

By Claim \ref{lwp-claim}(iii), $\Phi_{0,n}$ is classical.  By Claim \ref{lwp-claim}(v), $\Phi_{0,n}$ cannot be extended classically beyond $I_n$.  Thus by Claim \ref{lwp-claim}(iii) again (and uniqueness for classical wave maps), $\Phi_{0,n}: I_n \to \Energy$ is also the maximal Cauchy development for the initial data $\phi_n[t'_0]$ at time $t'_0$.  Applying Claim \ref{lwp-claim}(iv), we conclude that the $\Phi_{0,n}$ converge uniformly in $\Energy$ to $\phi'$ on any compact subinterval of $I'$.  This already shows that $\phi = \phi'$ on $I \cap I'$.  Furthermore, if a finite endpoint of $I$ lies in the interior of $I'$, then from Claim \ref{lwp-claim}(ii) we see that $\phi$ tends to a limit at this endpoint, contradicting Claim \ref{lwp-claim}(v); similarly if a finite endpoint of $I'$ lies in the interior of $I$.  Since $I$ and $I'$ both contain $t'_0$, we conclude that $I=I'$, as desired.
\end{proof}

In view of this lemma, we can now refer to maximal Cauchy developments or energy class solutions $\phi: I \to \Energy$ without having to specify any initial time $t_0 \in I$ or initial data $\phi[t_0]$.  Another consequence of the above lemma is that every energy class solution has a unique maximal extension, which is a maximal Cauchy development.  By Claim \ref{lwp-claim}(iii), there is also a well-defined notion of an energy class solution being \emph{classical}, since if it is classical at one time then it is classical at all times (though, as always, the classical solution corresponding to an energy solution is ambiguous up to a rotation \eqref{rotate}).

We also record a variant of this lemma:

\begin{lemma}[Closure of energy class solutions under uniform limits]\label{uniform}  Let $\phi_n: I_n \to \Energy$ be a sequence of energy class solutions, and let $K$ be a compact interval which is contained in $I_n$ for all sufficiently large $n$.  Suppose that $\phi_n$ converges uniformly in $\Energy$ on $K$ to some limit $\phi: K \to \Energy$.  Then $\phi$ is also an energy class solution.
\end{lemma}

\begin{proof} We can assume that $K$ has non-empty interior, since the claim follows immediately from Claim \ref{lwp-claim} otherwise.  Let $t_0$ be any interior point of $K$ and let $\tilde \phi: I \to \Energy$ be the maximal Cauchy development of $\phi[t_0]$ from time $t_0$.  From Claim \ref{lwp-claim}(iv) we see that $\tilde \phi=\phi$ on $I \cap K$.  If any finite endpoint of $I$ is contained in $K$, we see from the uniform convergence of $\phi_n$ to $K$ and Claim \ref{lwp-claim}(ii) that $\tilde \phi$ has a limit at this endpoint, contradicting Claim \ref{lwp-claim}(v); thus $K \subset I$, and the claim follows.	
\end{proof}

The symmetries \eqref{time-trans}, \eqref{space-trans}, \eqref{time-reverse}, \eqref{cov-scaling} clearly transform classical maximal Cauchy developments to other classical maximal Cauchy developments $\Time_{t_0} \phi: I+t_0 \to \S$, $\Trans_{x_0} \phi: I \to \S$, $\Rev \phi: -I \to \S$, $\Dil_\lambda: \lambda I \to \S$.  By limiting arguments using Claim \ref{energy-claim} and Lemma \ref{uniform}, we can extend these symmetries continuously to maximal Cauchy developments $\phi: I \to \Energy$ in the energy space.  It is clear that any group identities\footnote{One could of course unify all these symmetries together into a single Lie group if desired to make this claim more precise.} obeyed by the symmetries classically (e.g. the commutativity of time and space translation) will continue to hold in the energy space setting.

From Definition \ref{ap-def} we see that the above symmetries also transform (maximal) almost periodic solutions to (maximal) almost periodic solutions, with the frequency scale function and position function transforming by the rules
\begin{align}
N_{\Time_{t_0} \phi}(t) &:= N_\phi(t-t_0) \label{time-n}\\
x_{\Time_{t_0} \phi}(t) &:= x_\phi(t-t_0) \label{time-x}\\
N_{\Trans_{x_0} \phi}(t) &:= N_\phi(t) \label{trans-n}\\
x_{\Trans_{x_0} \phi}(t) &:= x_\phi(t) + x_0 \label{trans-x}\\
N_{\Rev \phi}(t) &:= N_\phi(-t) \label{rev-n}\\
x_{\Rev \phi}(t) &:= x_\phi(-t) \label{rev-x}\\
N_{\Dil_\lambda \phi}(t) &:= \frac{1}{\lambda} N_\phi(\frac{t}{\lambda}) \label{scale-n}\\
x_{\Dil_\lambda \phi}(t) &:= \lambda x_\phi(\frac{t}{\lambda}). \label{scale-x}
\end{align}
Furthermore, the compactness modulus $K$ of the transformed solutions is the same as that of the original solution (except in the case of time reversal, in which case it is a reflection of the original).

A basic construction we shall use repeatedly is the \emph{rescaling} $\phi_n: I_n \to \Energy$ of an almost periodic solution $\phi: I \to \Energy$ at some time $t_n \in I$, defined as
\begin{equation}\label{phin} \phi_n := \operatorname{Dil}_{N(t_n)} \operatorname{Time}_{-t_n} \operatorname{Trans}_{-x(t_n)} \phi,
\end{equation}
where $N: I \to (0,+\infty)$ and $x: I \to \R^2$ are the frequency scale function and the position function of $\phi$.  Note that $I_n := \{ t: t_n + t/N(t_n) \in I \}$, and that the frequency scale function $N_n: I_n \to (0,+\infty)$ and position function $x_n: I_n \to \R^2$ of $\phi_n$ are given by the formulae
\begin{equation}\label{phin-scale}
N_n(t) := \frac{ N(t_n + t/N(t_n) )}{N(t_n)}; \quad x_n(t) = N(t_n) (x(t_n+t/N(t_n)) - x(t_n)).
\end{equation}
In particular $0 \in I_n$ and
\begin{equation}\label{sort}
N_n(0) = 1; \quad x_n(0) = 0.
\end{equation}

As a first application of this rescaling, we show that the frequency scale function $N(t)$ blows up at any finite endpoint.

\begin{lemma}[Blowup of $N(t)$]\label{blow}  Let $\phi: I \to \Energy$ be an almost periodic maximal Cauchy development with frequency scale function $N: I \to (0,+\infty)$, and let $t_*$ be any finite endpoint of $I$.  Then $\liminf_{t \to t_*; t \in I} |t-t_*| N(t) > 0$.
\end{lemma}

\begin{proof} By the time reversal symmetry \eqref{time-reverse} we may assume that $t_* = \sup(I)$ is the upper endpoint of $I$.  Suppose for contradiction that the claim failed, then we can find a sequence $t_n \in I$ converging to $t_*$ such that $N(t_n) = o( 1 / (t_*-t_n) )$.  Then if we consider the rescaled solutions \eqref{phin}, we see that these are almost periodic maximal Cauchy development on an interval $I_n$ containing $0$ (but whose upper endpoint is converging to zero), with $\phi_n[0]$ lying in a compact subset of $\Energy$.  Thus, by passing to a subsequence, we may assume that $\phi_n[0]$ converges in $\Energy$ to a limit $\Phi_\infty$, which has a maximal Cauchy development $\phi_\infty$ on an open neighbourhood of $0$.  But this contradicts Claim \ref{lwp-claim}(iv) and the fact that the upper endpoint of $I_n$ converges to zero.
\end{proof}

Now we begin utilising the stress-energy tensor $\T$.

\begin{lemma}[Localisation of energy density]\label{localise}  Le $\phi: I \to \Energy$ be an almost periodic solution with frequency scale function $N: I \to (0,+\infty)$ and position function $x: I \to \R^2$.  Then for every $\eps > 0$ there exists $C > 0$ such that
$$ \int_{|x-x(t)| \geq C/N(t)} \T_{00}\ dx \leq \eps$$
for every $t \in I$, or equivalently that
$$ \int_{|x-x(t)| < C/N(t)} \T_{00}\ dx \geq \E(\phi) - \eps$$
for every $t \in I$.
\end{lemma}

\begin{proof}  By Definition \ref{ap-def} and Claim \ref{energy-claim}(iv), we see that the stress-energy tensor of $\operatorname{Dil}_{N(t)} \operatorname{Trans}_{-x(t)} \phi[t]$ for $t \in I$ lies inside a compact subset of $L^1(\R^2)$, and is thus tight in the sense that for every $\eps > 0$ there exists $C > 0$ such that
$$ \int_{|x| \geq C} \T_{00}( \operatorname{Dil}_{N(t)} \operatorname{Trans}_{-x(t)} \phi[t] )\ dx \leq \eps$$
for all $t \in I$.  Undoing the scaling, we obtain the claim.
\end{proof}

We can now show that $N$ and $x$ are stable on short time intervals:

\begin{lemma}[Quasicontinuity of $N(t)$ and $x(t)$]\label{quasi}  Let $\phi: I \to \Energy$ be an almost periodic solution with frequency scale function $N: I \to (0,+\infty)$ and scale function $x: I \to \R^2$ with non-zero energy.  Then there exists constants $c, C > 0$ such that $\frac{1}{C} N(t) \leq N(t') \leq C N(t)$ and $|x(t')-x(t)| \leq C/N(t)$ whenever $t,t' \in I$ are such that $|t-t'| \leq c/N(t)$.
\end{lemma}

\begin{proof} If this lemma failed, then we could find sequences of times $t_n, t'_n$ with $t'_n - t_n = o( 1/N(t_n) )$ such that $N(t'_n)/N(t_n)$ diverged either to infinity or to zero, or such that $N(t_n) |x(t')-x(t)|$ diverged to infinity.  If we let $\phi_n$ be the rescaled solutions \eqref{phin}, we thus see (by \eqref{phin-scale}) that there is a sequence of times $s_n \to 0$ such that $N_n(s_n)$ diverges either to infinity or to zero, or such that $x_n(s_n)$ diverges to infinity.  In particular, this shows (by Lemma \ref{localise}) that the stress energy tensors $\T( \phi_n[s_n] ) \in L^1(\R^2)$ have no convergent subsequence in $L^1(\R^2)$.

On the other hand, by passing to a subsequence as in the proof of Lemma \ref{blow} we may assume that $\phi_n$ converges uniformly on some neighbourhood $I_\infty$ of $0$ to some limit $\phi_\infty: I_\infty \to \Energy$.  By Claim \ref{lwp-claim}(ii), we conclude that $\phi_n[s_n]$ converges in $\Energy$, and hence by Claim \ref{energy-claim}(iv) $\T(\phi_n[s_n])$ converges in $L^1(\R^2)$, a contradiction.
\end{proof}

From Lemma \ref{quasi} and compactness we immediately obtain

\begin{corollary}[Local boundedness of $N(t)$ and $x(t)$]\label{localbound} Let $\phi: I \to \Energy$ be an almost periodic solution with frequency scale function $N: I \to (0,+\infty)$ and position function $x: I \to \R^2$.  Then for any compact subinterval $J$ of $I$ we have $0 < \inf_{t \in J} N(t) \leq \sup_{t \in J} N(t) < +\infty$ and $\sup_{t \in J} |x(t)| < +\infty$.
\end{corollary}

Now we establish a compactness property of almost periodic solutions.

\begin{lemma}[Compactness]\label{compaq}  Let $\phi_n: I_n \to \Energy$ be a sequence of almost periodic solutions with frequency scale function $N_n: I_n \to (0,+\infty)$ and scale function $x_n: I_n \to \R^2$, and all with the same compactness modulus $K$.  Let $I$ be an interval such that for every compact subinterval $J$ of $I$, that $I_n$ contains $J$ for all sufficiently large $n$, and that 
\begin{equation}\label{ntj}
 0 < \inf_{n; t \in J} N_n(t) \leq \sup_{n; t \in J} N_n(t) < \infty
\end{equation}
and
\begin{equation}\label{xnt}
 \sup_{n; t \in J} |x_n(t)| < \infty
\end{equation}
where the inf and sup are over $n$ that are sufficiently large depending on $J$.  Then after passing to a subsequence, $\phi_n$ converges uniformly in $\Energy$ on compact subintervals of $I$ to an almost periodic solution $\phi: I \to \Energy$ with frequency scale function $N: I \to (0,+\infty)$ and position function $x: I \to \R^2$ obeying the bounds
\begin{equation}\label{nn}
 \liminf_{n \to \infty} N_n(t) \leq N(t) \leq \limsup_{n \to \infty} N_n(t)
 \end{equation}
and
\begin{equation}\label{xx}
 \liminf_{n \to \infty} |x_n(t) - x(t)| = 0
 \end{equation}
for all $t \in I$, and with compactness modulus $K$.
\end{lemma}

\begin{proof}  Pick some time $t_0 \in I$.  By \eqref{ntj}, \eqref{xnt}, we may pass to a subsequence so that $N_n(t_0)$ and $x_n(t_0)$ are both convergent (with the former converging to a non-zero value).  Since the $\phi_n$ are almost periodic with the same compactness modulus $K$, and because the action of scaling and translation on $\Energy$ is continuous (Claim \ref{energy-claim}(iii)), this implies (after passing to a further subsequence) that $\phi_n[t_0]$ converges in $\Energy$ to some limit $\Phi_0$.

Let $\phi: I_* \to \Energy$ be the maximal Cauchy development from the initial data $\Phi_0$ at time $t_0$.  Then we see from Claim \ref{lwp-claim}(iv) that $\phi_n$ converges pointwise on $I \cap I_*$ in $\Energy$ to $\phi$.  Now suppose that some endpoint $t_*$ of $I_*$ lies in $I$. Let $J$ be a compact interval in $I \cap \overline{I_*}$ containing $t_*$ as an endpoint.  By hypothesis, $N_n(t)$ is bounded uniformly above and below on $J$, and $x_n(t)$ is also bounded, thus by Claim \ref{energy-claim}(iii) $\phi_n(t)$ ranges in a fixed compact subset of $\Energy$ for all $n$ and all $t \in J$; taking limits, we see that $\phi(t)$ ranges in the same compact set for $t \in J \cap I_*$.  But then we can find a sequence $t_n \in J \cap I_*$ converging to $t_*$ such that $\phi(t_n)$ converges in $\Energy$, contradicting Claim \ref{lwp-claim}(v).  Thus we see that no endpoint of $I_*$ can lie in $I$, and thus $I \subset I_*$.  Restricting $\phi$ to $I$, we obtain the desired solution $\phi: I \to \Energy$.

It remains to show that $\phi$ is almost periodic with the desired bounds.  Recall for each fixed $t \in I$ that $\phi_n$ converges in $\Energy$ to $\phi$.  By \eqref{ntj}, \eqref{xnt}, we may thus find a subsequence $n_{t,j}$ (depending on $t$) for which $N_{n_{t,j}}(t)$ and $x_{n_{t,j}}(t)$ converge to limits $N(t)$, $x(t)$ obeying \eqref{nn}, \eqref{xx}.  Since the $\phi_n$ were already almost periodic with compactness modulus $K$, the claim follows (again using the continuity from Claim \ref{energy-claim}(iii)).
\end{proof}

Now we use the conservation law \eqref{conserv}.

\begin{lemma}[Conservation identities]\label{conslemma}  Let $\phi: [T_-,T_+] \to \Energy$ be a energy class solution, and let $X^\alpha$ be a smooth vector field.  Let $\T$ be the stress-energy tensor of $\phi$.
\begin{itemize}
\item[(i)] If $X$ is compactly supported, then we have
\begin{equation}\label{tab2}
\int_{T_-}^{T_+} \int_{\R^2} \T_{\alpha \beta} \partial^\alpha X^\beta\ dx dt
= -\int_{\R^2} \T_{0 \beta} X^\beta\ dx|_{t=T_-}^{t=T_+}.
\end{equation}
\item[(ii)] If instead $\phi$ is a classical wave map, $x_0 \in \R^2$, and $t_0 > T_+$, then we have
\begin{equation}\label{stokes}
\begin{split}
\int_{T_-}^{T_+} \int_{|x-x_0| \leq |t-t_0|} \T_{\alpha \beta} \partial^\alpha X^\beta\ dx dt
&= -\int_{|x-x_0| \leq |t-t_0|} \T_{0 \beta} X^\beta\ dx|_{t=T_-}^{t=T_+}\\
&\quad -\int_{T_-}^{T_+} \int_{|x-x_0|=|t-t_0|} \T_{\alpha\beta} X^\beta L^\alpha\ d\sigma dt
\end{split}
\end{equation}
where $d\sigma$ is surface measure on the circle $\{ x \in \R^2: |x-x_0| = |t-t_0| \}$, and $L^\alpha$ is the inward null vector field $L = \partial_t - \frac{x-x_0}{|x-x_0|} \cdot \nabla_x$.
\end{itemize}
\end{lemma}

\begin{proof}  For classical wave maps, the claims follow immediately from \eqref{conserv} and Stokes' theorem (or integration by parts).  For claim (i) for energy class solutions, the claim then follows by a limiting argument using Claim \ref{energy-claim}(i), (iv) and Claim \ref{lwp-claim}(iii).
\end{proof}

\begin{corollary}[Finite speed of propagation]\label{finprop}  Let $\phi: [T_-,T_+] \to \Energy$ be a energy class solution with stress-energy tensor $\phi$.  Then
$$ \int_{|x-x_0| \leq |T_+ - t_0|} \T_{00}(T_+,x)\ dx \leq \int_{|x-x_0| \leq |T_- - t_0|} \T_{00}(T_-,x)\ dx $$
or equivalently
$$ \int_{|x-x_0| > |T_+ - t_0|} \T_{00}(T_+,x)\ dx \geq \int_{|x-x_0| > |T_- - t_0|} \T_{00}(T_-,x)\ dx $$
for all $x_0 \in \R^2$ and $t_0 > T_+$.
\end{corollary}

\begin{proof} By limiting arguments as in Lemma \ref{conslemma} we can take $\phi$ to be a classical wave map.  We now apply \eqref{stokes} with $X := \partial_t$ and observe (from \eqref{stress-def}) that the energy flux $\T_{\alpha\beta} X^\beta L^\alpha$ is non-negative.
\end{proof}

As a corollary we obtain an approximate Lipschitz property on $x(t)$ (generalising the claim in Lemma \ref{quasi}):

\begin{lemma}[Lipschitz nature of $x(t)$]\label{xlip}  Let $\phi: I \to \Energy$ be an almost periodic solution with frequency scale function $N: I \to (0,+\infty)$ and scale function $x: I \to \R^2$ with non-zero energy.  Then there exists a constant $C > 0$ such that $|x(t)-x(t')| \leq |t-t'| + C(1/N(t) + 1/N(t'))$ for all $t, t' \in I$.
\end{lemma}

\begin{proof} Let $E > 0$ denote the energy of $\phi$.  By Lemma \ref{localise} there exists a constant $C > 0$ such that
$$ \int_{|x-x(t)| \geq C/N(t)} \T_{00}(\phi[t])(x)\ dx < E/2$$
for all $t \in I$, and thus also
$$ \int_{|x-x(t)| < C/N(t)} \T_{00}(\phi[t])(x)\ dx \geq E/2.$$
Applying Corollary \ref{finprop} we conclude
$$ \int_{|x-x(t)| \geq C/N(t) + |t-t'|} \T_{00}(\phi[t'])(x)\ dx \leq \int_{|x-x(t)| \geq C/N(t)} \T_{00}(\phi[t])(x)\ dx < E/2$$
for any $t, t' \in I$.  Thus the ball $\{ x: |x-x(t')| < C/N(t') \}$ cannot be contained in $|x-x(t)| \geq C/N(t) + |t-t'|$, and the claim follows.
\end{proof}

\section{Normalised ancient solutions}

We now begin the proof of Theorem \ref{main}.  In this section we assume that Claim \ref{energy-claim}, Claim \ref{lwp-claim}, and Claim \ref{minimal} holds.

We first formalise a definition from the introduction.

\begin{definition}[Normalised ancient solution]  A \emph{normalised ancient solution} is an almost periodic partial Cauchy development $\phi: (-\infty,0] \to \Energy$ with frequency scale and position functions $N: (-\infty,0] \to (0,+\infty)$, $x: (-\infty,0] \to \R^2$ such that $N(0)=1$, $x(0)=0$, and $N(t) \leq 1$ for all $t \leq 0$.
\end{definition}

\begin{remark} Note that such solutions can be continued a little bit beyond time $t=0$ thanks to the contrapositive of Lemma \ref{blow}. However, we will not use this continuation here.
\end{remark}
In this section we show

\begin{proposition}\label{normas}  Suppose that Conjecture \ref{conj2} fails.  Then there exists a normalised ancient solution with non-zero energy.
\end{proposition}

\begin{proof} Applying Claim \ref{minimal}, we can find an almost periodic solution $\phi: I \to \Energy$ of non-zero energy.  By time translation \eqref{time-trans} we may assume $0 \in I$, thus $0 < \sup(I) \leq +\infty$.

Suppose first that $\sup_{0 \leq t < \sup(I)} N(t) < \infty$, thus $N$ is bounded from above by some constant $C$.  Then by Lemma \ref{blow} we have $\sup(I) = +\infty$, and the claim follows by applying time reversal \eqref{time-reverse} and modifying $N$ by a constant factor (which does not disrupt the almost periodicity, thanks to Claim \ref{energy-claim}(iii)).

Now suppose instead that $\sup_{0 \leq t < \sup(I)} N(t) = +\infty$.  From Lemma \ref{quasi} and Corollary \ref{localbound} one can find a sequence $t_n \to \sup(I)$ of times in $(0,\sup(I)) \subset I$ and a constant $C > 0$ such that $N(t_n) \to \infty$ and
$$ \sup_{0 \leq t < t_n} N(t) \leq C N(t_n)$$
(basically, one is selecting $t_n$ to be the ``current world record'' times in which $N(t_n)$ has essentially exceeded all priori values of $N$).

If we then let $\phi_n: I_n \to \Energy$ be the rescaled solutions \eqref{phin}, then we see that $0 \in I_n$, that $\inf(I_n) \to -\infty$, and that
$$ \sup_{-T_n \leq t \leq 0} N_n(t) \leq C$$
for some sequence of times $T_n \to +\infty$.  Applying \eqref{sort} and Corollary \ref{localbound} (which applies uniformly to the $\phi_n$ by rescaling from $\phi$) we also see that
$$ \liminf_{n \to \infty} \inf_{-T \leq t \leq 0} |N_n(t)| > 0$$
and
$$ \limsup_{n \to \infty} \sup_{-T \leq t \leq 0} |x_n(t)| < \infty$$
for every $T > 0$.  (In fact, one has more quantitative bounds of the form $|N_n(t)| \geq c/(1+|t|)$ and $|x_n(t)| \leq C(1+|t|)$ for some $c, C > 0$, although we will not need these sharper bounds.)

We can now apply Lemma \ref{compaq} and conclude (after passing to a subsequence) that the $\phi_n$ converge (in the sense of that lemma) to an almost periodic solution $\tilde \phi: (-\infty,0] \to \Energy$ with $\tilde N(0)=1$, $\tilde x(0)=0$, and $\tilde N(t) \leq C$ for all $t \leq 0$.  By rescaling $N$ (and adjusting the compactness modulus $K$ appropriately) one may assume $N(t) \leq 1$ for all $t \leq 0$.  Thus $\tilde \phi$ is a normalised ancient solution.

Finally, since $\phi_n$ converges pointwise to $\tilde \phi$ in $\Energy$, we see (from Claim \ref{energy-claim}(iv)) that the energies $\E(\phi_n)$ converge to $\E(\tilde \phi)$.  On the other hand, since energy is scale invariant for classical data (and hence in the space $\Energy$ also, by limiting arguments) we have $\E(\phi_n) = \E(\phi) > 0$ for all $n$.  Thus $\E(\tilde \phi)$ has non-zero energy as claimed.
\end{proof}

\section{A Morawetz estimate}

We now present a key tool in our analysis, namely a Morawetz estimate for wave maps that shows that such maps tend to be self-similar on the average around the spacetime origin $(0,0)$ (or indeed at any other point in spacetime).  

\begin{proposition}[Morawetz estimate]\label{moraw}  Let $[T_-,T_+] \subset (-\infty,0)$ be a time interval, and let $\phi: [T_-,T_+] \to \Energy$ be a partial Cauchy development.  Then we have
$$ \int_{T_-}^{T_+} \int_{|x| \leq 2|t|} |\partial_t \phi + \frac{x}{t} \cdot \nabla_x \phi|_{h(\phi)}^2\ \frac{dx dt}{|t|}
\leq C \sqrt{\log \frac{|T_-|}{|T_+|}} \E(\phi)$$
for some absolute constant $C > 0$, where we interpret the integrand in terms of the stress-energy tensor using \eqref{destress}.
\end{proposition}

\begin{remark} Note that if one crudely bounds $|\partial_t \phi + \frac{x}{t} \cdot \nabla_x \phi|_{h(\phi)}^2$
as $O( \T_{00})$ then one only obtains a bound of $O( \log \frac{|T_-|}{|T_+|} \E(\phi) )$; thus the non-trivial content of this estimate lies in the square root of the right-hand side, in the regime when $|T_-|$ is much larger than $|T_+|$.  The estimate here is essentially due to Grillakis \cite{grillakis-energy} (see also a variant in \cite{tao:forges}), but we provide a self-contained proof here.
\end{remark}

\begin{proof}  By a limiting argument using Claim \ref{lwp-claim} and Claim \ref{energy-claim} we see that it suffices to prove this claim for classical wave maps.    By the above remark we may also assume $|T_-| \geq 2 |T_+|$.

Applying \eqref{stokes} with $X$ set equal to the time vector field $\partial_t$ and using the non-negativity of the energy density $\T_{00}$ we obtain the flux bound
\begin{equation}\label{fluxbound}
\int_{T_-}^{T_+} \int_{|x|=|t|} \T_{L0}\ d\sigma dt \leq \E(\phi) 
\end{equation}
where we abbreviate $\T_{L\beta} := \T_{\alpha \beta} L^\alpha$.
Next, we apply \eqref{stokes} with $X$ set equal to the vector field $X^\alpha := x^\alpha / \rho_\eps$, where $\rho_\eps := \sqrt{ (1+\eps) t^2 - x^2 }$ and $0 \leq \eps \leq 1$ is to be chosen later.  Observe that $X$ is smooth all components of $X^\beta$ are $O(\eps^{-1/2})$ when $|x|\leq |t|$.  Using the easily verified bound $\T_{L\beta} = O( \T_{L0} )$ for each component $\beta$, we conclude from \eqref{fluxbound} that
$$ \int_{T_-}^{T_+} \int_{|x|=|t|} \T_{L\beta} X^\beta\ d\sigma dt = O( \eps^{-1/2} \E(\phi ) ).$$
A similar argument also gives
$$ \int_{T_-}^{T_+} \int_{|x|=|t|} \T_{L\beta} X^\beta\ d\sigma dt = O( \eps^{-1/2} \E(\phi) )$$
and thus the right-hand side of \eqref{stokes} here is $O( \eps^{-1/2} \E(\phi) )$.

On the other hand, we have
$$ \partial^\alpha X^\beta = \frac{g^{\alpha \beta}}{\rho_\eps} - \frac{x^\beta \partial^\alpha \rho_\eps}{\rho_\eps^2}$$
and hence by \eqref{stress-def} (and the fact that spacetime is three-dimensional)
$$
\T_{\alpha \beta} \partial^\alpha X^\beta = \frac{\langle \partial^\gamma \phi, \partial_\gamma \phi \rangle_{h(\phi)}}{2\rho_\eps^2} (x^\alpha \partial_\alpha \rho_\eps - \rho_\eps) - \frac{\langle x^\beta \partial_\beta \phi, \partial^\alpha \rho_\eps \partial_\alpha \phi \rangle_{h(\phi)}}{\rho_\eps^2}.$$
But as $\rho_\eps$ is homogeneous of degree $1$, we have $x^\alpha \partial_\alpha \rho_\eps = \rho_\eps$.  We conclude from \eqref{stokes} that
$$
-\int_{T_-}^{T_+} \int_{|x| \leq |t|} \frac{\langle x^\beta \partial_\beta \phi, \partial^\alpha \rho_\eps \partial_\alpha \phi \rangle_{h(\phi)}}{\rho_\eps^2} \ dx dt = O( \eps^{-1/2} \E(\phi) ).$$
Now observe that
$$ -\partial^\alpha \rho_\eps \partial_\alpha \phi = \frac{1}{\rho_\eps} ( x^\alpha \partial_\alpha \phi + \eps t \partial_t \phi )$$
and thus by Cauchy-Schwarz
$$ -\langle x^\beta \partial_\beta \phi, \partial^\alpha \rho_\eps \partial_\alpha \phi \rangle_{h(\phi)}
\geq \frac{1}{\rho_\eps} ( \frac{1}{2} |x^\alpha \partial_\alpha \phi|_{h(\phi)}^2 - O( \eps^2 |t|^2 |\partial_t \phi|_{h(\phi)}|^2 ) )$$
(say).  We thus have
$$
\int_{T_-}^{T_+} \int_{|x| \leq |t|} \frac{|x^\alpha \partial_\alpha \phi|_{h(\phi)}^2}{\rho_\eps^3} \ dx dt \leq O( \eps^{-1/2} \E(\phi) )
+ O( \eps^2 \int_{T_-}^{T_+} \int_{|x| \leq |t|} \frac{|t|^2}{\rho_\eps^3} |\partial_t \phi|_{h(\phi)}|^2\ dx dt).$$
Using the crude bounds $\frac{t^2}{\rho_\eps^3} = O( \eps^{-3/2} t^{-1}$ and $|\partial_t \phi|_{h(\phi)}|^2 = O( \T_{00} )$, the second term on the right-hand side is bounded by $O( \eps^{1/2} \log \frac{T_+}{T_-} \E(\phi) )$.  As for the left-hand side, we use the crude bound $\rho_\eps = O( |t| )$ and conclude that
$$ \int_{T_-}^{T_+} \int_{|x| \leq |t|} |\partial_t \phi + \frac{x}{t} \cdot \nabla_x \phi|_{h(\phi)}^2\ \frac{dx dt}{|t|}
\leq O( \eps^{-1/2} + \eps^{1/2} \log \frac{T_+}{T_-} ) \E(\phi).$$
Optimising in $\eps$ one obtains the claim in the interior region $\{ (x,t): |x| \leq |t| \}$ of the light cone.

Now we turn to the outer region $\{ (x,t): |t| \leq |x| \leq 2|t| \}$.  By polar coordinates it suffices to show that
$$ \int_{T_-}^{T_+} \int_{|x|=c|t|} |\partial_t \phi + \frac{x}{t} \cdot \nabla_x \phi|_{h(\phi)}^2\ d\sigma dt
\leq C \E(\phi)$$
uniformly for all $1 \leq c \leq 2$. 

Fix $c$.  By a slight modification of \eqref{stokes} (using the cone $\{ |x| = c|t| \}$ instead of $\{ |x| = |t|\}$) with $X := c \partial_t - \frac{x}{|x|} \cdot \nabla_x$ we have
$$ \int_{T_-}^{T_+} \int_{|x|=c|t|} c \T_{00} + \frac{x_i}{|x|} \T_{0i}\ d\sigma dt
= - \int_{|x| \leq c|t|} \T_{00}\ dx|_{t=T_-}^{t=T_+} \leq \E(\phi).$$
Note from \eqref{stress-def} that the integrand on the left-hand side is non-negative for $c \geq 1$.
On the other hand, from \eqref{stress-def} we have
$$ |\partial_t \phi + \frac{x}{|x|} \cdot \nabla_x \phi|_{h(\phi)}^2 \leq C ( \T_{00} + \frac{x_i}{|x|} \T_{0i} )$$
for some absolute constant $C > 0$, and hence by the triangle inequality
$$ |\partial_t \phi + \frac{x}{t} \cdot \nabla_x \phi|_{h(\phi)}^2 \leq C_1 ( \T_{00} + \frac{x_i}{|x|} \T_{0i} ) + C_2 (c-1)^2 \T_{00} \leq C_3 ( c \T_{00} + \frac{x_i}{|x|} \T_{0i} )$$
for some absolute constants $C_1, C_2, C_3 > 0$. The claim then follows.
\end{proof}

This leads to the following corollary.  

\begin{definition}[Asymptotic self-similarity]  A Cauchy development $\phi: (-\infty,0] \to \Energy$ is said to be $\phi$ is \emph{asymptotically self-similar} along a sequence of times $t_n \to -\infty$ if we have
\begin{equation}\label{fix0}
 \lim_{n \to \infty} \int_{\R^2} |\partial_t \phi + \frac{x}{t_n} \cdot \nabla_x \phi|_{h(\phi)}^2(t_n,x)\ dx = 0
\end{equation}
and
\begin{equation}\label{fix1}
\lim_{n \to \infty} \int_{|x| \geq (1+\eps) |t_n|} \T_{00}(t_n,x)\ dx = 0
\end{equation}
for all $\eps > 0$, and similarly that
\begin{equation}\label{ph0}
\lim_{n \to \infty} \sup_{1 < A \leq 1/\eps} \frac{1}{2\log A} \int_{A t_n}^{t_n/A} \int_{\R^2} |\partial_t \phi + \frac{x}{t} \cdot \nabla_x \phi|_{h(\phi)}^2\ \frac{dx dt}{|t|} = 0
\end{equation}
and
\begin{equation}\label{ph1}
\lim_{n \to \infty} \sup_{1 < A \leq 1/\eps} \frac{1}{2 \log A} \int_{A t_n}^{t_n/A} \int_{|x| \geq (1+\eps) |t|} \T_{00}\ \frac{dx dt}{|t|} = 0
\end{equation}
for all $\eps > 0$.
\end{definition}

\begin{corollary}  Let $\phi: (-\infty,0] \to \Energy$ be a Cauchy development.  Then $\phi$ is asymptotically self-similar along at least one sequence of times $t_n \to -\infty$.
\end{corollary}

\begin{proof}  For each $\eps, \delta > 0$, let $\Omega_{\eps,\delta}$ be the set of times $T \in (-\infty,0)$ for which one has
$$ \sup_{0 < A < 1/\eps} \frac{1}{2\log A} \int_{AT}^{T/A} \int_{|x| \leq 2t} |\partial_t \phi + \frac{x}{t} \cdot \nabla_x \phi|_{h(\phi)}^2\ \frac{dx dt}{|t|} \geq \delta.$$
From Proposition \ref{moraw} we have
$$ \lim_{T \to +\infty} \frac{1}{\log T} \int_{-T}^{-1} \int_{|x| \leq 2t} |\partial_t \phi + \frac{x}{t} \cdot \nabla_x \phi|_{h(\phi)}^2\ dx \frac{dt}{|t|} = 0.$$
Applying the Hardy-Littlewood maximal inequality 
$$
|\{ s \in \R: \sup_{r > 0} \frac{1}{2r} \int_{s-r}^{s+r} |f(s')|\ ds' > \lambda \}| \leq \frac{C}{\lambda} \int_\R |f(s)|\ ds$$
to the logarithmic time variable $s := \log |t|$, we conclude that $\Omega_{\eps,\delta}$ has logarithmic density zero, or in other words that 
$$ \lim_{T \to +\infty} \frac{1}{\log T} \int_{-T}^{-1} 1_{\Omega_{\eps,\delta}}(t) \frac{dt}{|t|} = 0.$$
By taking a countable sequence of $\eps, \delta$ going to zero and using the axiom of countable choice we can thus find a sequence $t_n \to -\infty$ such that
$$
\lim_{n \to \infty} \sup_{1 < A \leq 1/\eps} \frac{1}{2\log A} \int_{At_n}^{t_n/A} \int_{|x| \leq 2|t|} |\partial_t \phi + \frac{x}{t} \cdot \nabla_x \phi|_{h(\phi)}^2\ \frac{dx dt}{|t|} = 0$$
for all $\eps > 0$.  A similar argument lets us also ensure that
$$
\lim_{n \to \infty} \int_{|x| \leq 2|t_n|} |\partial_t \phi + \frac{x}{t} \cdot \nabla_x \phi|_{h(\phi)}^2\ dx = 0.$$
This will establish \eqref{fix0} and \eqref{ph0} once we have \eqref{fix1} and \eqref{ph1}.

To establish \eqref{ph1}, observe from Corollary \ref{finprop} that
$$ \int_{|x| > (1+\eps)|t|} \T_{00}(t_n,x)\ dx \leq \int_{|x| > \eps|t|} \T_{00}(0,x)\ dx$$
and thus
$$ \sup_{1 < A \leq 1/\eps} \frac{1}{2\log A} \int_{A t_n}^{t_n/A} \int_{|x| > (1+\eps)|t|} |\partial_t \phi + \frac{x}{t} \cdot \nabla_x \phi|_{h(\phi)}^2\ \frac{dx dt}{|t|} \leq \frac{C}{\eps} \int_{|x| \geq \eps^2 |t_n|} \T_{00}(0,x)\ dx.$$
By monotone (or dominated) convergence, the right-hand side goes to zero as $n \to \infty$.  This gives \eqref{ph1}.  The proof of \eqref{fix1} is similar.  Thus $\phi$ is asymptotically self-similar along $t_n$ as required.
\end{proof}

Now we specialise to the case of a normalised ancient solution $\phi$.  From Lemma \ref{quasi} and the normalisation $N(0)=1$, we see that $N(t) \geq c/(1+|t|)$ for all $t < 0$ and some $c > 0$; by Lemma \ref{xlip} we thus have $|x(t)| \leq C(1+|t|)$ for all $t<0$ and some $C > 0$.  In particular, by refining $t_n$ to a subsequence we may assume that $1/(|t_n| N(t_n))$ and $x(t_n)/t_n$ converge to some limits $\alpha$ and $v$ respectively.  Applying Proposition \ref{normas}, we thus obtain

\begin{corollary}\label{normas2}  Suppose that Conjecture \ref{conj2} fails.  Then there exists a normalised ancient solution with non-zero energy which is asymptotically self-similar along at least one sequence of times $t_n \to -\infty$.  Furthermore, we have $0 \leq \alpha < \infty$ and $v \in \R^2$ such that 
\begin{equation}\label{ntn}
1/N(t_n) = (\alpha+o(1)) |t_n|
\end{equation}
and 
\begin{equation}\label{xtn}
x(t_n) = (v+o(1))t_n.
\end{equation} 
\end{corollary}

The remainder of the paper is devoted to showing that the conclusion of Corollary \ref{normas2} is impossible for various values of $\alpha$ and $v$. Henceforth we shall assume Claim \ref{selfsim} and Claim \ref{nondeg}.

\section{The self-similar case}

The objective of this section is to show

\begin{proposition}[No self-similar solutions]\label{noss}  Let the notation and assumptions be as in Corollary \ref{normas2}.  Then it is not possible for $\alpha$ to be non-zero.
\end{proposition}

\begin{proof} Suppose for contradiction that $\alpha$ is non-zero.
Now consider the rescaled solutions $\phi_n: (-\infty,0] \to \Energy$ defined by $\phi_n := \Dil_{1/|t_n|}(\phi)$.  By \eqref{ntn}, \eqref{xtn}, the frequency scale function $N_n(t)$ and position function $x_n(t)$ of $\phi_n$ then obey the initial bounds
\begin{equation}\label{nn-init}
 N_n(-1) = \alpha + o(1); \quad x_n(-1) = -v + o(1).
\end{equation}
From the bounds $N(t) \geq c/(1+|t|)$, $|x(t)| \leq C(1+|t|)$ we also see that
$$ N_n(t) \geq c/|t| + o(1); \quad |x_n(t)| \leq C|t| + o(1)$$
for any fixed $t < 0$.

From \eqref{nn-init} we see that $\phi_n[-1]$ lies in a precompact subset of $\Energy$, so by passing to a subsequence we may assume that $\phi_n[-1]$ is convergent to some limit $\tilde \phi[-1]$, which has energy $\E(\phi) > 0$.  In particular, the stress-energies $\T(\phi_n[-1])$ converge in $L^1$ to $\T(\tilde \phi[-1])$, with $\int \T_{00}(\tilde \phi[-1]) = \E(\phi)$.

In order to be able to apply Lemma \ref{compaq}, we need some lower bounds on $N_n(t)$ for times $t$ other than $-1$.  We shall do this by analysing the stress-energy tensor of $\phi_n$.  From \eqref{ph0}, \eqref{ph1} and rescaling we know that
\begin{equation}\label{phon}
 \int_{-1/\eps}^{-\eps} \int_{\R^2} |t (\phi_n)_t + x \cdot \nabla_x \phi_n|_{h(\phi_n)}^2\ dx dt = o(1)
\end{equation}
and
\begin{equation}\label{phon2}
 \int_{-1/\eps}^{-\eps} \int_{|x| \geq (1+\eps)|t|} \T_{00}(\phi_n[t])\ dx dt = o(1)
\end{equation}
for any $\eps > 0$.

Morally speaking, these estimates are asserting that $\phi_n$ is approaching a self-similar solution as $n \to \infty$.  To make this precise we will need to use the conservation law \eqref{conserv} to ensure that the stress-energy tensor of $\phi_n$ is indeed behaving in a self-similar manner.

We turn to the details.  Fix $0 < \eps < 1$, and let $\eta \in C^\infty_0(\R^2)$ be a test function supported in the ball $\{ x: |x| \leq 1-\eps \}$, which we extend homogeneously to $(-\infty,0) \times \R^2$ by defining $\eta(t,x) := \eta(x/|t|)$.  We then apply \eqref{tab2} with $X^\alpha$ equal to the vector field $X^\alpha := \eta x^\alpha / \rho$, where $\rho := \sqrt{t^2-|x|^2}$.  A direct computation (noting that $x^\alpha \partial_\alpha \eta = 0$) shows that
$$ \T_{\alpha \beta}(\phi_n[t]) \partial^\alpha X^\beta = 
\langle \frac{x^\alpha}{\rho} \partial_\alpha \phi_n, \partial^\beta \eta \partial_\beta \phi_n \rangle_{h(\phi_n)} + \eta | \frac{x^\alpha}{\rho} \partial_\alpha \phi_n |_{h(\phi)}^2$$
(where the right-hand side is interpreted using \eqref{destress}, as usual).
Using \eqref{phon}, the bounded energy of $\phi_n$, and Cauchy-Schwarz we conclude that
\begin{equation}\label{sap}
 \int_{\R^2} \T_{0 \beta}(\phi_n[t]) X^\beta\ dx|_{t=T_-}^{t=T_+} = o(1)
\end{equation}
for any fixed $T_-, T_+$ and $\eta$, with the decay rate $o(1)$ being uniform when $T_-, T_+$ range in a compact subset of $(-\infty,0)$.  

Now recall that $\T(\phi_n[-1])$ was converging in $L^1$ to a non-zero function $\T(\tilde \phi[-1])$.  From \eqref{fix1} and rescaling we have
$$
\int_{|x| \geq (1+\eps)} \T_{00}(\phi_n[-1])\ dx = o(1)
$$
for any $\eps > 0$, and thus $\T_{00}(\tilde \phi[-1])$ vanishes outside of the ball $\{ x: |x| \leq 1 \}$.  

Let $0 < \eps_0 < 1$ be arbitrary, and let $0 < \eps_3 < \eps_2 < \eps_1 < \eps_0$ be small parameters to be chosen later.
Since $\T_{00}(\tilde \phi[-1])$ has total mass $\E(\phi) > 0$, we thus have
\begin{equation}\label{xe1}
 \int_{|x| > 1-3\eps_1} \T_{00}(\tilde \phi[-1])\ dx \leq \frac{1}{3} \E(\phi)
\end{equation}
(say), or equivalently
$$ \int_{|x| \leq 1-3\eps_1} \T_{00}(\tilde \phi[-1])\ dx \geq \frac{2}{3} \E(\phi)$$
if $\eps_1$ is small enough.

Since absolutely integrable functions are uniformly integrable, we have
$$ \int_{|x-x_0| \leq \eps_3} \T_{00}(\tilde \phi[-1])\ dx \leq \eps_2 \E(\phi)$$
for all $x_0 \in \R^2$, and in particular when $|x_0| \leq 1-2\eps_1$, if $\eps_3$ is sufficiently small depending on $\eps_2$, and thus
$$ \int_{|x-x_0| \leq \eps_3} \T_{00}(\phi_n[-1])\ dx \leq \eps_2 \E(\phi) + o(1).$$
Applying \eqref{sap} for a suitable choice of $\eta$ we conclude (if $\eps_2$ is small enough depending on $\eps_1$) that
$$ \int_{|x-x_0| \leq \eps_3/2} \T_{00}(\phi_n[t])\ dx \leq \frac{1}{3} \E(\phi) + o(1)$$
whenever $-1/\eps_0 \leq t \leq \eps_0$ and $|x_0| \leq 1-2\eps_1$; here we use the fact that we are working strictly within the light cone to control $\T_{00}$ by $\T_{0 \beta}(\phi_n[t]) x^\beta/\rho$.  Observe that we can make the decay rate $o(1)$ uniform in the choice of $t$ and $x_0$.  Combining this with \eqref{xe1} we conclude that
$$ \sup_{x_0 \in \R^2} \int_{|x-x_0| \leq \eps_3/2} \T_{00}(\phi_n[t])\ dx \leq \frac{1}{3} \E(\phi) + o(1)$$
Comparing this with Lemma \ref{localise} we conclude that $N_n(t)$ is bounded uniformly from below for $-1/\eps_0 \leq t \leq \eps_0$.

We can now apply Proposition \ref{compaq} and assume (after passing to a subsequence) that $\phi_n$ converges uniformly in $\Energy$ on $(-\infty,0)$ to an almost periodic Cauchy development $\tilde \phi: (-\infty,0) \to \Energy$ with energy $\E(\phi)$.  From \eqref{phon}, \eqref{phon2}, and Definition \ref{travel} we see that $\tilde \phi$ is self-similar.  But this contradicts Claim \ref{selfsim}(ii), and the claim follows.
\end{proof}

\section{The non-self-similar timelike case}

The objective of this section is to establish

\begin{proposition}[No timelike travelling solutions]\label{notravel}  Let the notation and assumptions be as in Corollary \ref{normas2}.  Then it is not possible for $\alpha$ to be zero with $|v| < 1$.
\end{proposition}

\begin{proof}  The argument here shall be similar to that of Proposition \ref{noss}.

By hypothesis, $N(t_n) = o(|t_n|)$.
Now consider the rescaled solutions $\phi_n: I_n \to \Energy$ defined by \eqref{phin}, then every compact time interval is contained in $I_n$ for sufficiently large $n$.  From \eqref{ph0} and rescaling we have
$$ \int_{-T}^{T} \int_{\R^2} |(\phi_n)_t + \frac{x_n(t_n)+x/N(t_n)}{t_n+t/N(t_n)} \cdot \nabla_x \phi_n|_{h(\phi_n)}^2\ dx dt = o(1)$$
for every $T > 0$ (and sufficiently large $n$).  From \eqref{ntn}, \eqref{xtn}, and the hypothesis $\alpha=0$, we thus conclude (using the fact that $\phi_n$ has bounded energy) that
\begin{equation}\label{intr}
 \int_{-T}^{T} \int_{|x| \leq R} |{\bf v}^\alpha \partial_\alpha \phi_n|_{h(\phi_n)}^2\ dx dt = o(1)
 \end{equation}
for all $T, R > 0$, where ${\bf v} \in \R^{1+2}$ is the four-velocity (or more accurately, three-velocity) ${\bf v} := (1, v)$.  

Morally, \eqref{intr} is asserting that $\phi_n$ is approaching a travelling wave with velocity $v$.  As in the previous section, in order to make this rigorous one must first show that the stress-energy tensor (or at least some key component of this tensor) is also travelling at velocity $v$ in the limit.

We turn to the details.  Let $\eta \in C^\infty_0(\R^2)$ be a bump function, which we extend to $\R^{1+2}$ as a travelling wave $\eta(t,x) := \eta(x-vt)$.  Applying \eqref{tab2} with $X^\alpha := \eta {\bf v}^\alpha$ we have
$$
\int_{T_-}^{T_+} \int_{\R^2} \T_{\alpha \beta}(\phi_n[t]) (\partial^\alpha \eta) {\bf v}^\beta\ dx dt
= -\int_{\R^2} \T_{0 \beta}(\phi_n[t]) \eta {\bf v}^\beta\ dx|_{t=T_-}^{t=T_+}.$$
Since $\partial^\alpha \eta \partial_\alpha {\bf v} = 0$, we see from Cauchy-Schwarz that
$$ |\T_{\alpha \beta}(\phi_n[t]) (\partial^\alpha \eta) {\bf v}^\beta|
\leq C_{v,\eta} |{\bf v}^\alpha \partial_\alpha \phi_n|_{h(\phi_n)} \T_00(\phi_n[t])^{1/2}.$$
Applying \eqref{intr} (and the bounded energy of $\phi_n$) we conclude that
\begin{equation}\label{etab}
 \int_{\R^2} \T_{0 \beta}(\phi_n[t]) \eta {\bf v}^\beta\ dx|_{t=T_-}^{t=T_+} = o(1)
 \end{equation}
for any fixed $T_-, T_+, \eta$.

From \eqref{sort} we have $N_n(0)=1$ and $x_n(0)=0$.  By arguing exactly as in Proposition \ref{noss} (using \eqref{etab} in place of \eqref{sap}) we conclude that $N_n(t)$ is bounded from below on any given time interval $[-T,T]$ and uniformly in $n$.  (Here we use the hypothesis $|v| < 1$ to control $\T_{00}$ by $\T_{0\beta} {\bf v}^\beta$.)  Applying Lemma \ref{quasi} we see that it is also bounded from above on this time interval, as is $x_n(t)$.  We may thus invoke Proposition \ref{compaq} and (after passing to a subsequence) assume that the $\phi_n$ converge uniformly in $\Energy$ on compact intervals to an almost periodic solution $\tilde \phi: \R \to \Energy$ with energy $\E(\phi)$.  From \eqref{intr} and Definition \ref{travel} we see that $\tilde \phi$ is a travelling wave with velocity $v$.  But this contradicts Claim \ref{selfsim}(i), and the claim follows.
\end{proof}

\section{The non-self-similar lightlike case}

To conclude the proof of Theorem \ref{main}, we now show

\begin{proposition}[No lightlike travelling solutions]\label{nolight}  Let the notation and assumptions be as in Corollary \ref{normas2}.  Then it is not possible for $\alpha$ to be zero with $|v| \geq 1$.
\end{proposition}

\begin{proof}
From \eqref{ntn}, Lemma \ref{xlip} and the normalisation $N(0)=1$, $x(0)=0$ we have $|x(t_n)| \leq (1+o(1)) |t_n|$; comparing this with \eqref{xtn} we conclude $|v| \leq 1$.  Since we also have $|v| \geq 1$, we conclude $|v|=1$.  To simplify notation we shall take $v = e_1$ to be the standard basis vector; the general case is of course similar.

Let $\phi_n: I_n \to \Energy$ be the rescaled solutions \eqref{phin}, and let ${\bf v} := (1,e_1)$ be the null vector associated to $v=e_1$.   Then by repeating the arguments used to prove Proposition \ref{notravel}, we have the bound \eqref{intr} for any $T, R > 0$, and \eqref{etab} for any $T_-, T_+, \eta$.  By repeating the derivation of \eqref{intr} (using \eqref{fix0} instead of \eqref{ph0}) we also have
$$ \int_{|x| \leq R} |{\bf v}^\alpha \partial_\alpha \phi_n|_{h(\phi_n)}^2\ dx dt = o(1)$$
for every $R > 0$, and hence (by Lemma \ref{localise}
\begin{equation}\label{locale}
\int_{\R^2} |{\bf v}^\alpha \partial_\alpha \phi_n|_{h(\phi_n)}^2\ dx = o(1)
\end{equation}

Note for classical wave maps $\phi$ that the vanishing of ${\bf v}^\alpha \partial_\alpha \phi = \partial_t \phi + \partial_1 \phi$ would imply (from \eqref{cov}) that $(\phi^* \nabla)_2 \partial_2 \phi = 0$, which by integration by parts (and the fact that $\phi$ is constant outside of a compact set) would imply that $\int_{\R^2} |\partial_2 \phi|_{h(\phi)}^2\ dx = 0$.  The purpose of the manipulations below is to replicate this observation for energy class solutions, using the stress-energy tensor as a substitute for \eqref{cov} (which is no longer directly available).

We turn to the details.  We introduce the dual null vector $\overline{\bf v} := (1,-e_1)$, and apply \eqref{tab2} with $X^\alpha := \eta \overline{\bf v}^\alpha$ and some bump function $\eta \in C^\infty_0(\R^2)$, extended to $\R^{1+2}$ as $\eta(t,x) := \eta(x-t e_1)$ as before.  We conclude
$$
\int_{T_-}^{T_+} \int_{\R^2} \T_{\alpha \beta}(\phi_n[t]) (\partial^\alpha \eta) \overline{{\bf v}}^\beta\ dx dt
= -\int_{\R^2} \T_{0 \beta}(\phi_n[t]) \eta \overline{{\bf v}}^\beta\ dx|_{t=T_-}^{t=T_+}.$$
By monotone convergence, this identity also holds for functions $\eta(t,x) = \varphi(x_1-t)$ for bump functions $\varphi \in C^\infty_0(\R)$.

Let $\varphi$ and $\eta$ be as above.  The right-hand side can be bounded in magnitude by $C_\varphi \E(\phi)$ for some constant $C_\varphi$ depending only on $\varphi$.  As for the left-hand side, we use \eqref{stress-def} to observe the identity
$$
\T_{\alpha \beta}(\phi_n[t]) (\partial^\alpha \eta) \overline{{\bf v}}^\beta = 2 (\partial_1 \eta) \T_{0 \beta}(\phi_n[t]) {\bf v}^\beta - (\partial_1 \eta) |{\bf v}^\alpha \partial_\alpha \phi_n|_{h(\phi_n)}^2.$$
Applying \eqref{intr}, we conclude that
$$
|\int_{T_-}^{T_+} \int_{\R^2} 2 (\partial_1 \eta) \T_{0 \beta}(\phi_n[t]) {\bf v}^\beta \ dx dt|
\leq C_\varphi \E(\phi) + o(1).$$
On the other hand, from \eqref{etab} we have
$$ \int_{T_-}^{T_+} \int_{\R^2} 2 (\partial_1 \eta) \T_{0 \beta}(\phi_n[t]) {\bf v}^\beta \ dx dt = 
(T_+-T_-) \int_{\R^2} 2 (\partial_1 \eta) \T_{0 \beta}(\phi_n[0]) {\bf v}^\beta \ dx + o(1)$$
and thus
$$ |\int_{\R^2} 2 (\partial_1 \eta) \T_{0 \beta}(\phi_n[0]) {\bf v}^\beta \ dx| \leq \frac{C_\varphi}{T_+-T_-} \E(\phi) + o(1).$$
Letting $T_+-T_-$ be arbitrarily large, we conclude that
$$ \int_{\R^2} 2 (\partial_1 \eta) \T_{0 \beta}(\phi_n[0]) {\bf v}^\beta \ dx = o(1).$$
On the other hand, from \eqref{stress-def} we observe that
$$ 2 (\partial_1 \eta) \T_{0 \beta}(\phi_n[0]) {\bf v}^\beta = |{\bf v}^\alpha \partial_\alpha \phi_n(0,x)|_{h(\phi_n)}^2 + |\partial_2 \phi_n(0,x)|_{h(\phi_n)}^2$$
and thus by \eqref{locale} we have
$$ \int_{\R^2} (\partial_1 \eta) |\partial_2 \phi_n(0,x)|_{h(\phi_n)}^2\ dx = o(1).$$
Now let $0 < r < R$ be positive numbers, and apply the above fact with $\varphi$ set equal to a smooth function supported on $[-r, R+r]$ which equals $1$ on $[r, R]$, and is increasing on $[-r,r]$ with derivative at least $1/10r$ on $[-r/2,r/2]$ and decreasing with derivative $O(1/r)$ on $[R,R+r]$. We conclude that
$$ \int_{|x| \leq r/2} |\partial_2 \phi_n(0,x)|_{h(\phi_n)}^2\ dx \leq C_r \int_{|x| \geq R} |\partial_2 \phi_n|_{h(\phi_n)}^2 + o(1).$$
On the other hand, from Lemma \ref{localise} we can make the first term on the right-hand side arbitrarily small by increasing $R$ (keeping $r$ fixed), and thus
$$ \int_{|x| \leq r/2} |\partial_2 \phi_n(0,x)|_{h(\phi_n)}^2\ dx  = o(1).$$
Using Lemma \ref{localise} again we conclude
\begin{equation}\label{ph2n}
 \int_{\R^2} |\partial_2 \phi_n(0,x)|_{h(\phi_n)}^2\ dx  = o(1).
\end{equation}

Now $\phi_n[0]$ ranges in a compact subset of $\Energy$, and thus by passing to a subsequence converges in $\Energy$ to some limit $\Phi = (\phi_0,\phi_1)$.  Since the $\phi_n$ have energy $\E(\phi)$, $\Phi$ does also; in particular, $\Phi$ has non-zero energy.  From \eqref{locale} we have
$$ \int_{\R^2} |\phi_1 + \partial_1 \phi_0|_{h(\phi_0)}^2\ dx = 0$$
while from \eqref{ph2n} we have
$$ \int_{\R^2} |\partial_2 \phi_0|_{h(\phi_0)}^2\ dx = 0.$$
But this now contradicts Claim \ref{nondeg}, and the claim follows.
\end{proof}

\end{document}